\documentclass[11pt]{amsart}

\usepackage[a4paper,hmargin=3.5cm,vmargin=4cm]{geometry}
\usepackage{amsfonts,amssymb,amscd,amstext}
\usepackage{graphicx,float}
\usepackage[dvips]{epsfig}
\usepackage{pstricks,pst-plot}
\usepackage{hyperref}

\usepackage{fancyhdr}
\input xy
\xyoption{all}

\usepackage{enumerate}
\usepackage{titlesec}
\usepackage{mathrsfs}

\titleformat{\section}%[display]
{\filcenter\bfseries\large} {\thesection{.}}{0.2cm}{}%[$\vspace*{-1.0cm}$]
\titleformat{\subsection}[runin]
{\bfseries} {\thesubsection{.}}{0.15cm}{}[.]
\titleformat{\subsubsection}[runin]
{\em}{\thesubsubsection{.}}{0.15cm}{}[.]
\usepackage[up,bf]{caption}
\newtheorem{theorem}{Theorem}[section]
\newtheorem{proposition}[theorem]{Proposition}
\newtheorem{claim}[theorem]{Claim}
\newtheorem{lemma}[theorem]{Lemma}
\newtheorem{corollary}[theorem]{Corollary}

\theoremstyle{definition}
\newtheorem{definition}[theorem]{Definition}
\newtheorem{remark}[theorem]{Remark}

\numberwithin{equation}{section}
\numberwithin{figure}{section}

\newcommand\Hcal{\mathcal{H}}

\newcommand\Gcal{\mathcal{G}}
\newcommand\Ocal{\mathcal{O}}
\newcommand\Pcal{\mathcal{P}}
\newcommand\Rcal{\mathcal{R}}

\newcommand\Fcal{\mathcal{F}}

\newcommand\Cscr{\mathscr{C}}

\newcommand\Zscr{\mathscr{Z}}

\newcommand\Q{\mathbb{Q}}

\def\c{\mathbb{C}}

\def\n{\mathbb{N}}
\renewcommand\r{\mathbb{R}}
\newcommand\s{\mathbb{S}}

\newcommand\z{\mathbb{Z}}
\renewcommand\k{\mathbb{K}}

\newcommand\igot{\mathfrak{i}}

\renewcommand\igot{\mathfrak{i}}

\newcommand\pgot{\mathfrak{p}}

\newcommand\Agot{\mathfrak{A}}

\newcommand\Mgot{\mathfrak{M}}

\newcommand\Ygot{\mathfrak{Y}}

\newcommand\I{\imath}
\renewcommand\imath{\igot}

\newcommand\Psf{\mathsf{P}}

\newcommand\wt{\widetilde}

\newcommand\di{\partial}

\newcommand\Flux{\mathrm{Flux}}

\newcommand\supp{\mathrm{supp}}

\newcommand\Div{\mathfrak{Div}}

\def\Spin{{\rm Spin}}

\usepackage{color}

\begin{document}

\fancyhead[LO]{Interpolation and optimal hitting for minimal surfaces} 
\fancyhead[RE]{A.\ Alarc\'on, I.\ Castro-Infantes,  and F.\ J.\ L\'opez}
\fancyhead[RO,LE]{\thepage}

\thispagestyle{empty}

%% Title
\vspace*{7mm}
\begin{center}
{\bf \LARGE Interpolation and optimal hitting for complete minimal surfaces with finite total curvature}

\vspace*{5mm}

%% Authors
{\large\bf Antonio Alarc\'on, Ildefonso Castro-Infantes, and Francisco J.\ L\'opez}
\end{center}

\vspace*{7mm}

\begin{quote}
{\small
\noindent {\bf Abstract} \hspace*{0.1cm}
We prove that, given a compact Riemann surface $\Sigma$ and disjoint finite sets $\varnothing\neq E\subset\Sigma$ and $\Lambda\subset\Sigma$, every map $\Lambda \to \r^3$ extends to a complete conformal minimal immersion $\Sigma\setminus E\to \r^3$ with finite total curvature. 

%We obtain this result as a consequence of a more precise one providing approximation, interpolation of given finite order, and control on the flux.

This result opens the door to  study  optimal hitting problems in the framework of complete minimal surfaces in $\r^3$ with finite total curvature. To this respect we provide, for each integer $r\ge 1$, a set $A\subset\r^3$ consisting of  $12r+3$ points in an affine plane such that if $A$ is contained in a complete nonflat orientable immersed minimal surface $X\colon M\to\r^3$, then the absolute value of the total curvature of $X$ is greater than $4\pi r$. In order to prove this result we obtain an upper bound for the number of intersections of a complete immersed minimal surface of finite total curvature in $\r^3$ with a straight line not contained in it, in terms of the total curvature and the Euler characteristic of the surface.
\medskip

\noindent{\bf Keywords} \hspace*{0.1cm} 
minimal surface, finite total curvature, Riemann surface, meromorphic function, interpolation theory, optimal hitting.

\medskip

\noindent{\bf MSC (2010)} \hspace*{0.1cm} 
53A10, %Minimal surfaces
52C42, %Immersions (minimal, prescribed curvature, tight, etc.)
30D30, % Meromorphic functions, general theory
32E30. %Holomorphic and polynomial approximation, Runge pairs, interpolation
}
\end{quote}

%%%%%%%%%%
%%%%%%%%%%
%%%%%%%%%%
%%%%%%%%%%  Section: Introduction
%%%%%%%%%%
%%%%%%%%%%

\section{Introduction}\label{sec:intro}

Complete minimal surfaces of finite total curvature have been one of the main focus of interest in the  global theory of minimal surfaces in $\r^3$; we refer for instance to \cite{Osserman-book,BarbosaColares1986LNM,Yang1994MA,LopezMartin1999PM} for background on the topic. This subject is intimately related to the one of meromorphic functions and $1$-forms on compact Riemann surfaces. Indeed, if $M$ is an open Riemann surface and $X\colon M\to \r^3$ is a complete conformal minimal immersion with finite total curvature, then there are a compact Riemann surface $\Sigma$ and a finite set $\varnothing\neq E\subset \Sigma$ such that $M$ is biholomorphic to $\Sigma\setminus E$ and the exterior derivative $d X$ of $X\colon\Sigma\setminus E\to\r^3$, which coincides with its $(1,0)$-part $\di X$ since $X$ is harmonic,  is holomorphic and extends meromorphically to $\Sigma$ with effective poles at all points in $E$ (i.e., it can not be finite at any end). In particular, the Gauss map $\Sigma\setminus E\to \s^2$ of $X$ extends conformally to $\Sigma$ and, up to composing with the stereographic projection, is a meromorphic function on $\Sigma$. 
Any orientable complete minimal surface in $\r^3$ of finite total curvature comes in this way (see Osserman \cite{Osserman-book}).

Polynomial Interpolation is a fundamental subject in mapping theory. 
Given disjoint finite sets $E\neq\varnothing$ and $\Lambda$ in a compact Riemann surface $\Sigma$, the classical Riemann-Roch theorem enables to prescribe the values on $\Lambda$ of a meromorphic function $\Sigma\to \c$ that is holomorphic in $\Sigma\setminus E$. Our first main result is an analogue for complete minimal surfaces in $\r^3$ with finite total curvature of this  result.

\begin{theorem}\label{th:intro}
Let $\Sigma$ be a compact Riemann surface with empty boundary and let $E\neq \varnothing$ and $\Lambda$ be disjoint finite sets in $\Sigma$. Every map $\Lambda \to \r^3$ extends to a complete conformal minimal immersion $\Sigma\setminus E\to \r^3$ with finite total curvature.
\end{theorem}
Somewhat surprisingly, no result of this type seems to be available in the literature, even in the specially simple case when $\Sigma\setminus E=\c$. Only recently Alarc\'on and Castro-Infantes initiated in \cite{AlarconCastro2017} the theory of interpolation by conformal minimal surfaces in $\r^n$, for arbitrary dimension $n\ge 3$, but their construction method does not ensure any control on the total curvature of the interpolating surfaces. We emphasize that those with finite total curvature form a tiny subset in the space of all complete conformal minimal immersions $\Sigma\setminus E\to\r^3$ (the fact that this subset is nonempty for every $\Sigma$ and $E$  was observed by Pirola in \cite{Pirola1998PJM}).
 We also point out that the finiteness of the set $\Lambda$ cannot be deleted from the assumptions of Theorem \ref{th:intro}. For instance, there is no complete nonflat minimal surface in $\r^3$ with finite total curvature and containing the set $\z^2\times \{0\}\subset \r^3$ (see Corollary \ref{co:Z3} for more general sets); this shows that the theorem fails for infinite sets $\Lambda$, even under the natural assumptions that $\Lambda$ is closed and discrete in $\Sigma\setminus E$ and that the prescription of values $\Lambda\to\r^3$ is proper.

We shall obtain Theorem \ref{th:intro} as a consequence of a more precise result providing approximation, interpolation of given finite order, and control on the flux (see Theorem \ref{th:main-1}). 
In the last few years, we have witnessed the birth and development of the approximation theory for minimal surfaces in $\r^n$ $(n\ge 3)$. A crucial point for this array of results was the observation that the punctured null quadric $\Agot_*=\{z=(z_1,\ldots,z_n)\in\c^n\colon z_1^2+\cdots+z_n^2=0\}\setminus\{0\}$ directing minimal surfaces in $\r^n$ is a complex homogeneous manifold, and hence an Oka manifold (see \cite[Example 4.4]{AlarconForstneric2014IM}; we refer to L\'arusson \cite{Larusson2010NAMS} for a brief introduction and to Forstneri\v c \cite{Forstneric2017} for a complete monograph on Oka theory). In particular, for any open Riemann surface $M$, there are {\em many} holomorphic maps $M\to \Agot_*$, and hence, up to solving the period problem,  also {\em many} conformal minimal immersions $M\to\r^n$. The Runge-Mergelyan theorem for conformal minimal immersions $M\to \r^n$ (see Alarc\'on and L\'opez \cite{AlarconLopez2012JDG,AlarconLopez2015GT} and Alarc\'on, Forstneri\v c, and L\'opez \cite{AlarconForstnericLopez2016MZ,AlarconForstnericLopez2016Pre1}) has given rise to plenty of applications (see \cite{AlarconForstneric2017Survey} for an up-to-date survey of results in this direction). However, the holomorphic flexibility of $\Agot_*$ does not extend to the algebraic category; this is why the construction methods developed in the above sources do not provide complete minimal surfaces in $\r^n$ with finite total curvature. Only in $\r^3$ and exploiting the {\em spinorial representation} for minimal surfaces in that dimension, L\'opez \cite{Lopez2014TAMS} was able to prove a Runge-Mergelyan type uniform approximation theorem for complete minimal surfaces with finite total curvature. 
The main new ingredient in our method of proof is the construction of holomorphic {\em sprays of algebraic spinorial representations} of minimal surfaces in $\r^3$; this enables us to combine the ideas in \cite{AlarconCastro2017} and \cite{Lopez2014TAMS} in order to ensure, simultaneously, approximation and interpolation of finite order. The use of holomorphic sprays in the study of minimal surfaces in the Euclidean spaces was introduced by Alarc\'on and Forstneri\v c in \cite{AlarconForstneric2014IM}.

Theorem \ref{th:intro} opens the door to a new line of research, namely, the study of {\em optimal hitting} problems in the framework of complete minimal surfaces in $\r^3$ with finite total curvature. We now discuss this novel subject.

Unless otherwise stated, in the remainder we shall only consider surfaces with empty boundary.
Write ${\rm TC}(X)$ for the total curvature of a complete orientable immersed minimal surface $X\colon M\to \r^3$, and recall that, by the classical Osserman's theorem, ${\rm TC}(X)$ is a nonnegative integer multiple of $-4\pi$ (see Section \ref{ss:FTC} or \cite{Osserman-book}). 
\begin{definition} \label{de:espacios}
For any integer $r\ge 1$ we denote by $\Zscr_r$ the space of all complete nonflat immersed orientable minimal surfaces $X\colon M\to \r^3$  with empty boundary and $|{\rm TC}(X)|\le 4 \pi r$.  Likewise, for any integer $m\le 1$  we write $\Zscr_{r;m}$ for the subset of $\Zscr_r$ consisting of those surfaces $X\colon M\to \r^3$  with  the Euler characteristic $\chi(M)=m$.
\end{definition}
Recall that, by the Jorge-Meeks formula \eqref{eq:j-m} (see  \cite{JorgeMeeks1983T}), $\Zscr_{r;m}\neq\varnothing$ forces $2r + m\geq 2$. Thus
\[
\Zscr_r=\bigcup_{m\le 1}\Zscr_{r;m}\; =\bigcup_{m\in \{2-2 r,\ldots,1\}}\Zscr_{r;m}.
\]

A simple observation is that, up to homotheties and rigid motions, every complete minimal surface of finite total curvature contains any given set $A\subset\r^3$ consisting of at most $3$ points. If the set $A$ consists of $4$ points not contained in a plane, it is not hard to see that there is an Enneper's surface which contains $A$.
In general, in view of Theorem \ref{th:intro}, given a finite set $\varnothing\neq A\subset\r^3$ and an integer $m\le 1$, there is a {\em large enough} integer $r\geq 1$  such that $A\subset X(M)$ for some immersion $X\in\Zscr_{r;m}$.  Therefore, fixed integers $r\geq 1$ and $m\leq 1$ with $2r + m\geq 2$, one wonders whether there exist finite sets $A\subset\r^3$ which are {\em against} the family $\Zscr_{r;m}$, meaning that $A$ is contained in the image of no immersion $X\in\Zscr_{r;m}$, and, if such sets exist, what is the cardinal of the smaller ones. The same questions arise for the bigger family $\Zscr_r$ for any $r\ge 1$.
The second main result of this paper may be stated as follows.
 
%
% Theorem \ref{th:main-intro-2}
%
 \begin{theorem}\label{th:main-intro-2}
Let $r\ge 1$ be an integer. For any integer $m$ with  $2-2r \le m\le 1$  there is a set $A_{r;m}\subset\r^3$ which is against the family  $\bigcup_{k\le m}\Zscr_{r;k}$ and consists of $12r +2m+1$ points whose affine span is a plane. In particular, the set $A_{r;1}$, which consists of $12r+3$ points, is against the family $\Zscr_r$. 

Thus, if $X\colon M\to\r^3$ is a complete nonflat orientable immersed minimal surface with empty boundary and  $\chi(M)\le m$, and  if $A_{r;m}\subset X(M)$, then the total curvature ${\rm TC}(X)<-4\pi r$. In particular, no complete nonflat orientable immersed minimal surface $X$ with $|{\rm TC}(X)|\le 4\pi r$ contains $A_{r;1}$.
\end{theorem}

Notice that in the above theorem we have $13\le 12r +2m+1\le 12r+3$. Recall, for instance, that $\Zscr_1$ consists precisely of all catenoids and Enneper's surfaces (see Osserman \cite{Osserman-book}); and hence Theorem \ref{th:main-intro-2} furnishes a set $A\subset\r^3$, consisting of $15$ points, which is contained in no catenoid and in no Enneper's surface (this particular number does not seem to be sharp). In general, it remains an open question whether the bounds $12r +2m+1$ and $12r +3$ given by Theorem \ref{th:main-intro-2}  are sharp or not;   determining the optimal ones (which, by Theorem \ref{th:intro}, do  exist!) is, likely, a highly nontrivial task.

Theorem \ref{th:main-intro-2} is, to the best of the author's knowledge, the first result of its type within the theory of minimal surfaces. 
In order to prove it we shall provide, given a complete minimal surface $X\colon M\to\r^3$ of finite total curvature, an upper bound on the cardinal of the preimage by $X$ of any affine line in $\r^3$ not contained in $X(M)$. To be precise, we show the following.
\begin{theorem} \label{th:against}
Let $X\colon M\to \r^3$ be a complete orientable immersed minimal surface with finite total curvature and empty boundary. If $L\subset\r^3$ is a straight line  not contained in $X(M)$, then
\[
	\# \big(X^{-1}(L)\big)\leq 6 {\rm Deg}(N)+\chi(M)=-\frac3{2\pi}{\rm TC}(X)+\chi(M),
\]
where ${\rm Deg}(N)$ is the degree of the Gauss map $N$ of $X$.
\end{theorem}
Despite a bound like the one in Theorem \ref{th:main-intro-2} was expected, no explicit one is known in the literature. Hence,  from the point of view of the classical theory of minimal surfaces in $\r^3$, this theorem can be viewed as the main result of the present paper.

%
% Organization
%

\subsection*{Organization of the paper}
In Section \ref{sec:prelim} we state the background which is required for the well understanding of the rest of the paper. We prove Theorem \ref{th:intro} on interpolation in Section \ref{sec:interpolation} and Theorem \ref{th:main-intro-2} on optimal hitting in Section \ref{sec:hitting}.

%%%%%%%%%%
%%%%%%%%%%
%%%%%%%%%% Section: Preliminaries
%%%%%%%%%%   
%%%%%%%%%%
%%%%%%%%%%

\section{Preliminaries}\label{sec:prelim}

We denote $\imath=\sqrt{-1}$ and $\z_+=\{0,1,2,\ldots\}$.
Given an integer $n\in\n=\{1,2,3,\ldots\}$ and $\k\in\{\r,\c\}$, we denote by $|\cdot|$ %and $\dist(\cdot,\cdot)$
the Euclidean norm %and distance
in $\k^n$.%, respectively. %If $K$ is a compact topological space and $f\colon K\to \k^n$ is a continuous map, we denote by 

\subsection{Riemann surfaces and spaces of maps}\label{ss:RS}

Throughout the paper every Riemann surface will be considered connected if the contrary is not indicated.

Let $M$ be an open Riemann surface. Given a subset $A\subseteq M$ we denote by $\Cscr(A,\k^n)$ the space of continuous functions $A\to\k^n$, by $\Ocal(A)$ (respectively, $\Mgot(A)$) the space of functions $A\to \c$ which are holomorphic (respectively, meromorphic) on an unspecified open neighborhood of $A$ in $M$; functions in $\Ocal(A)$ (respectively, in $\Mgot(A)$)  will be called holomorphic (respectively, meromorphic) on $A$.  Likewise, by a {\em conformal minimal immersion} $A\to \r^n$ we mean the restriction to $A$ of a  conformal minimal immersion on an open neighborhood of $A$ in $M$.

%By a {\em compact bordered Riemann surface} we mean a compact Riemann surface $M$ with nonempty boundary $bM$ consisting of finitely many pairwise disjoint smooth Jordan curves. The interior $\mathring M=M\setminus bM$ of $M$ is called a {\em bordered Riemann surface}. It is well known that every compact bordered Riemann surface $M$ is diffeomorphic to a smoothly bounded compact domain in an open Riemann surface $\wt M$. The spaces $\Ascr^r(M)$ and $\Ascr^r(M,Z)$, for an integer $r\in\z_+$ and a complex manifold $Z$, are defined as above.

%An open Riemann surface $M$ is said to be of {\em finite conformal type} if $M=\oM\setminus\{q_1,\ldots,q_r\}$, where $\oM$ is a compact Riemann surface without boundary and  $\{q_1,\ldots,q_r\}\subset \oM$.  

A compact subset $K$ in an open Riemann surface $M$ is said to be {\em Runge} (also called {\em holomorphically convex} or {\em $\Ocal(M)$-convex}) if every continuous function $K\to\c$, holomorphic in the interior $\mathring K$, may be approximated uniformly on $K$ by holomorphic functions on $M$; by the Runge-Mergelyan theorem \cite{Runge1885AM,Mergelyan1951DAN,Bishop1958PJM} this is equivalent to that $M\setminus K$ has no relatively compact connected components in $M$. 

\begin{definition}\label{def:fct}
Let $\Sigma$ be a compact Riemann surface with possibly non empty  boundary $b\Sigma$ and $E\subset \Sigma\setminus b\Sigma$ be a finite subset. By definition, $M:=\Sigma\setminus E$ is said to be a Riemann surface of {\em finite conformal type}. If $E\neq \varnothing$, the points in $E$ will be called   the topological ends of $M$.
\end{definition}

%%%%%%%%%%
%%%%%%%%%%
%%%%%%%%%% Subsection: Spinors
%%%%%%%%%%   
%%%%%%%%%%

\subsection{Spinorial $1$-forms on  Riemann surfaces}

Let $M$ be a  Riemann surface with empty boundary. For $W\subset M$, we denote by $\Div(W)$ the free commutative group of {\em countable} divisors of $W$ with multiplicative notation, that is to say,
\[
\Div(W)=\{\prod_{j\in \n}q_j^{n_j}\colon q_j\in W,\; n_j\in \z \}.
\]
If $D=\prod_{j\in \n}q_j^{n_j}\in\Div(W)$,  the set $\{q_j\colon n_j\neq 0\}$ is said to be the {\em support} of $D$ (as usual, $q^0=1$ for all $q\in W$). A divisor $\prod_{j\in \n}q_j^{n_j}\in\Div(W)$ is said to be {\em integral} if $n_j\geq 0$ for all $j\in \n$.  

We denote by $\Omega(M)$ and $\Omega_m(M)$ the spaces of holomorphic and meromorphic $1$-forms on $M$, respectively.
If  $f\in \Mgot(M)$ is different from $0$, we denote by $(f)_0$ and $(f)_\infty$ the integral divisors of zeros and poles of $f$ in $M$ respectively, and $(f):=\frac{(f)_0}{(f)_\infty}$ the divisor of $f$ in $M$.  Likewise we define the divisors of non-zero meromorphic $1$-forms on $M$; if $M$ is compact, these divisors have finite support and are called {\em canonical}.

 %%% Spinorial functions and 1-forms

A holomorphic $1$-form $\omega$ is said to be {\em spinorial} if $(\omega)=D^2$ for some integral divisor $D\in\Div(M)$, equivalently, if the zeros   of $\omega$ have even order. We denote by $\Ygot(M)$ the set of holomorphic spinorial $1$-forms on $M$. We say that two $1$-forms $\omega_1$ and $\omega_2$ in $\Ygot(M)$ are   {\em spinorially equivalent}, and write $\omega_1\sim \omega_2$, if there exists a function $f\in\Mgot(M)$ such that $\omega_2=f^2 \omega_1$. An equivalence class $\Theta\in {\Ygot(M)}/{\sim}$ is called a spinorial structure on $M$. 

If $M$ is of finite topology and $k:=\dim_\z \Hcal_1(M,\z)$ is the dimension of the first homology group of $M$, there are $2^k$ pairwise distinct spinorial structures on $M$. Spinorial structures can be introduced from a topological point of view. Indeed,  take a class $\Theta$ in the set of spinorial structures and a $1$-form $\theta\in\Theta$. Consider an embedded closed curve $\gamma\subset M$, take an annulus $A$ being a tubular neighborhood of $\gamma$ in $M$, and choose a conformal parameter $z\colon A\to\{z\in\c\colon1<|z|<R \}$, ($R>1$). Then set $\xi_\Theta(\gamma):=0$ if $\sqrt{\theta(z)/dz}$ has a well defined branch on $A$ and $\xi_\Theta(\gamma)=1$ otherwise. The induced map $\xi_\Theta\colon \Hcal_1(M,\z)\to\z_2$ does not depend on the choice of $\theta\in\Theta$ and defines a group homomorphism. Further, $\xi_{\Theta_1}=\xi_{\Theta_2}$ if and only if $\Theta_1=\Theta_2$, and so, the set of spinorial structures may be identified with the set of group morphisms ${\rm Hom}(\Hcal_1(M,\z),\z_2)$. Furthermore, it is known that the map $\xi\colon \frac{\Ygot(M)}{\sim}\to{\rm Hom}(\Hcal_1(M,\z),\z_2)$ defined by $\xi(\Theta)=\xi_{\Theta}$ is bijective.

\begin{remark}\label{re:spinmero}
If $M=\Sigma\setminus E$ where $\Sigma$ is a compact Riemann surface and $E\subset \Sigma$ is a non-empty finite subset, any spinorial class $\Theta\in \Ygot(M)/\sim$ contains meromorphic $1$-forms on $\Sigma$, that is to say, there exists $\omega \in \Theta$ extending meromorphically to $\Sigma$. See \cite{Lopez2014TAMS} for details.
\end{remark}

\subsection{Spinorial data of a conformal minimal immersion}

Let $M$ be an open Riemann surface  and consider $X\colon M\to\r^3$ a conformal minimal immersion. Then, the $1$-forms $\phi_j:=\partial X_j$, $j=1,2,3$, are holomorphic on $M$ and satisfy $\sum_{j=1}^3 \phi_j^2=0$ and  $\sum_{j=1}^3 |\phi_j|^2\neq 0$. The meromorphic function  
\begin{equation}\label{eq:g}
g:=\frac{\phi_3}{\phi_1-\imath \phi_2}
\end{equation}
 defined on $M$ is the Gauss map of $X$ up to the stereographic projection.% and 
%\[
%\phi_1=\frac1{2} \big(\frac1g-g\big) \phi_3, \quad \phi_2=\frac{\I}{2} \big(\frac1g+g\big) \phi_3.
%\]
The pair $(g,\phi_3)$ is called the Weierstrass data of $X$. The holomorphic $1$-forms
\begin{equation}\label{eq:spindata}
\eta_1:=\frac{\phi_3}{g} \quad \text{ and }\quad \eta_2:=\phi_3 g
\end{equation}
%determine two holomorphic functions on $M$
%\begin{equation}\label{eq:w12}
%w_1:=\frac{\eta_1}{\theta} \qquad w_2:=\frac{\eta_2}{\theta}
%\end{equation}
are spinorial, spinorially equivalent, and have no common zeros on $M$. The pair $(\eta_1,\eta_2)$ defined on \eqref{eq:spindata} is said to be the {\em spinorial data} associated to the conformal minimal immersion $X$, or simply the {\em spinorial representation} of $X$. Thus, the pair $(\eta_1,\eta_2)$ determines a unique spinor structure $\Theta_X\in\frac{\Ygot(M)}{\sim}$ associated to the conformal minimal immersion $X$. 

Conversely, if $\eta_1$ and $\eta_2$ are two spinorial holomorphic $1$-forms on $M$ such that $\eta_1$ and $\eta_2$ are spinorially equivalent, $|\eta_1|+|\eta_2|$ never vanishes in $M$, and the holomorphic $1$-forms $\phi_1$, $\phi_2$, and $\phi_3$ defined by
\begin{equation}\label{eq:phiofeta}
\Psi(\eta_1,\eta_2):=(\phi_1,\phi_2,\phi_3)=\bigg(\frac12(\eta_1-\eta_2)
,\frac\I2(\eta_1+\eta_2),
\sqrt{\eta_1\eta_2}\bigg)
\end{equation}
have no real periods on $M$, then for the initial condition $X(p_0)\in\r^3$, the map $X\colon M\to\r^3$ defined by
\begin{equation}\label{eq:Xweier}
X(p)=X(p_0)+\Re \int_{p_0}^{p}\big(\phi_1,\phi_2,\phi_3\big), \quad p\in M
\end{equation}
is well defined and is a conformal minimal immersion with $\Theta_X$ the spinorial class of $\eta_j$, $j\in \{1,2\}$. Notice that $\phi_3$ is well defined up to the sign, hence $X$ up to a symmetry with respect to a horizontal plane. 
 
%We set ${\Spin}(M):= \{(\omega_1,\omega_2)\in \Ygot(M)^2\colon \eta_1\sim \eta_2 \;\text{and}\; (|\eta_1|+|\eta_2|)(p)\neq 0, \; p\in M\}$.

\subsection{Conformal minimal immersions of finite total curvature}\label{ss:FTC}

 Let $M$ be a Riemann surface. A conformal minimal immersion $X\colon M\to\r^3$ is said to be of {\em total finite curvature} or, acronymously, FTC if 
\[
{\rm TC}(X):=	\int_M K \ ds^2=-\int_M |K| \ ds^2>-\infty,
\]
where $K$ denotes the Gauss curvature of the conformal minimal immersion $X$.

Huber \cite{Huber1957CMH}  proved that every   complete Riemannian surface $(M,ds^2)$ with possibly non empty compact boundary  and $\int_M \min\{K,0\}ds^2>-\infty$  is of finite conformal type, see Definition \ref{def:fct}. This applies for any Riemann surface $M$ with possibly non empty compact boundary   which admits a conformal complete minimal immersion in $\r^3$ with FTC.    If $X\colon M\to \r^3$ is a complete conformal minimal immersion of FTC,   Osserman \cite{Osserman-book} proved that  the Weierstrass data $(g,\phi_3)$ of $X$, hence the $1$-forms $\phi_1,$ $\phi_2,$ and $\phi_3$, extend meromorphically to the Huber's compactification $\Sigma$ of $M$ (the same holds for the spinorial representation of $X$, $(\eta_1,\eta_2)$). Conversely, if  $M=\Sigma \setminus E$  is a Riemann surface of finite conformal type, $X\colon M\to \r^3$ is a   conformal minimal immersion, and the  Weierstrass data of $X$ extend meromorphically to the compactification $\Sigma$ of $M$ having effective poles (i.e., of positive order) at every end of $M$ (i.e., point in  $E$),  then $X$ is complete and of FTC. In any case, the Gauss map $N\colon M\to \s^2$ of $X$ extends conformally to $\Sigma$ and, when $bM=\varnothing$,  
\[
{\rm TC}(X)=-4 \pi {\rm Deg}(N),
\]
 where ${\rm Deg}(N)$ is the topological degree of $N\colon \Sigma\to \s^2$.

The asymptotic behavior of complete   minimal surfaces of FTC was studied by Jorge and Meeks in \cite{JorgeMeeks1983T}. They proved that any complete conformal minimal immersion $X\colon M=\Sigma\setminus E \to \r^3$ of FTC is a proper topological map.  Furthermore, if $E=\{q_1,\ldots,q_s\}\subset \Sigma$ denotes the set of ends  then there exists a Euclidean ball $B(R)=\{p\in \r^3\colon \|p \|<R\}$ of large enough radius $R>0$ so that:
\begin{itemize}
\item $X^{-1}(\r^3\setminus B(R))=\bigcup_{j=1}^s (D_j\setminus \{q_j\}),$
 where $D_1,\ldots,D_s$ are pairwise disjoint closed discs in $\Sigma$ and $ D_j$    contains $q_j$ as interior point for all $j=1,\ldots,s$. 
\smallskip
\item If $a_j=N(q_j)\in   \s^2$ denotes the limit normal vector at the end $q_j\in E$,   $\Pi_j:=\{p\in \r^3\colon \langle p, a_j\rangle =0\}$ is the limit tangent plane at $q_j$, $\pi_j\colon \r^3\to \Pi_j$ is the orthogonal projection, and $I_j+1\geq 2$ is the maximum of the pole orders    of the 1-forms  $\{\phi_k\colon   k=1,2,3\}$ at $q_j$, then 
 $(\pi_j\circ X)|_{D_j\setminus\{q_j\}}\colon D_j\setminus\{q_j\}\to \Pi_j\setminus  B(R)$ is an  $I_{j}$-sheeted proper  {\em sublinear multigraph (covering)}, $j=1,\ldots,s$.
\end{itemize}
In addition, the well known Jorge-Meeks formula holds:
\begin{equation} \label{eq:j-m}
2 {\rm Deg}(N)=-\chi(M)+\sum_{j=1}^s I_j=-\chi(\Sigma)+ \sum_{j=1}^s (I_j+1).
\end{equation}

\subsection{Some remarks on planar curves}\label{ss:cusp}

For any oriented closed curve $\gamma\subset \r^2$ and  point $p\notin \gamma$,    the {\em winding number} $w_\gamma(p)\in \z$ of   $\gamma $ with respect to  $p$   is the number of times the curve winds around $p$. If   we set $\gamma_p:=\frac{\gamma-p}{\|\gamma-p\|}\subset \s^1$ then  $w_\gamma(p)=w_{\gamma_p}(0,0)$, and this number is, up to the identification $\Hcal_1(\s^1,\z)=\z$ given by the topological degree,  the homology class $[\gamma_p]\in  \Hcal_1(\s^1,\z)$. If $\gamma$ is   sufficiently tame, it is well known that $w_\gamma(p)$    coincides with the number of {\em positive crossings} minus the number of {\em negative crossings} of any (oriented) ray   based at $p$ with $\gamma$.

  A curve $\gamma\subset \r^2$ is said to be {\em piecewise regular} if it is smooth and the tangent vector $\gamma'(t)$ vanishes only on a finite set of points, namely {\em singular points}.

	We say that a piecewise regular curve  $\gamma(t)$ admits a {\em regular} normal field if there exists  a regular smooth map $n\colon \gamma \to \s^1$ such that $\gamma'(t)\bot n(t):= n(\gamma(t))$ for all $t$. (Here, regular means that $n\colon \gamma \to \s^1$  is a local diffeomorphism.)  In this case, the {\em turning number} or  {\em rotation number} of $\gamma\geq 1$  is  the  number of rotations made by the   normal vector $n(t)$ 
 during one traversal of the curve, that is to say, the absolute value of the degree of $n$ (which coincides with   the total curvature of $\gamma$    divided by $2\pi$ when $\gamma$ is regular). Notice that the turning number does not depend on the chosen regular normal field.

 If $\gamma$ is piecewise regular and admits a regular normal field $n$, a singular point $\gamma(t_0)$ of $\gamma$ is said to be a {\em cusp point} if locally around $t_0$ we have 
\begin{equation} \label{eq:cusp}
\gamma(t)-\gamma(t_0) =(t-t_0)^{2 m} \beta(t),
\end{equation}
  where $m\in \n$, $\beta(t)$ is smooth, and $\beta(t_0)\neq 0$.  In this case the curve $\gamma(t)$ {\em turns back} at $t=t_0$, and any straight line being non orthogonal to $n(t_0)$ and  not containing $\gamma(t_0)$ locally intersects $\gamma$ around $\gamma(t_0)$ either at $0$ or $2$ points; see equation \eqref{eq:cusp}.

\begin{proposition}\label{pro:curve}
Let $\gamma(t)$ be a  smooth piecewise regular closed curve in $\r^2$ admitting a  regular normal field, and denote by $t_\gamma$ the turning number of $\gamma$. For any point $p\in \r^2\setminus \gamma$ we have $|w_\gamma(p)|\leq 2t_\gamma.$
 \end{proposition}
As it is shown in Figure \ref{fig:cus}, the bound in this proposition is sharp.
\begin{figure}[ht] 
    \begin{center}
 \scalebox{0.45}{\includegraphics{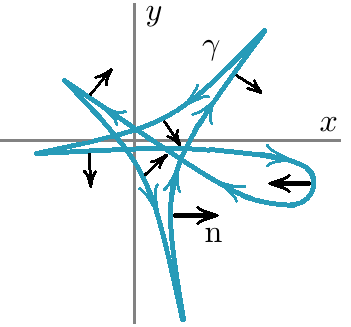}}
     \end{center}
\caption{An oriented planar curve $\gamma$ with four cusp points,  turning number $t_\gamma=1$, and winding number $w_\gamma=2$ with respect to the origin.}
\label{fig:cus}
\end{figure}

\begin{proof} Let $n(t)\colon \gamma\to \s^1$  be a regular normal field   along $\gamma(t)$, hence a local diffeomorphism. Call $\pi_0\colon \r^2\to \r$  the orthogonal projection $(x,y)\mapsto x$. Up to a rigid motion we can suppose that 
\begin{itemize}
\item the ray $\ell:=\{p+(0,y)\colon y\geq 0\}$   intersects  $\gamma$  only at regular points and in a transversal way.
\smallskip
\item $n^{-1}(\{(1,0),(-1,0)\})$ contains no singular point of $\gamma$.
\end{itemize}   
 Let $\gamma_1,\ldots, \gamma_{2 t_\gamma}$ denote the family of connected components of $ n^{-1}(\s^1_+)$ and  $n^{-1}(\s^1_-)$ ordered so that they are laid end to end, where 
$\s^1_+:=\{(x,y)\in \s^1\colon y\geq 0\}$ and $\s^1_-:=\{(x,y)\in \s^1\colon y\leq 0\}$.
Since each $\gamma_j$ is piecewise regular, we can split it into Jordan arcs $\gamma_{j,1}\ldots, \gamma_{j,m_j}$ lying end to end and satisfying:
\begin{itemize}
\item  the endpoints of  $\gamma_{j,i}$   are either cusp points or points in the set $n^{-1}(\{(1,0),(-1,0)\}$,
\smallskip
\item the interior  of  $\gamma_{j,i}$ contains no cusp points, $i=1,\ldots,m_j$.
\end{itemize}

An arc $\gamma_{j,i}$ is said to be {\em positive} if for any interior point $\gamma_{j,i}(t)$ of $\gamma_{j,i}$ , the ordered basis
  $\{(0,1), \gamma_{j,i}'(t)\}$ of $\r^2$ is positive, or equivalently, if the first coordinate function $(\pi_0\circ \gamma_{j,i})(t)$ on $\gamma_{j,i}$ is   decreasing. Otherwise $\gamma_{j,i}$ is said to be {\em negative}. Obviously, two consecutive sub-arcs $\gamma_{j,i}$ and $\gamma_{j,i+1}$ in $\gamma_j$ have different character; see \eqref{eq:cusp}. Since $\pi_0|_{\gamma_{j,i}}$ is one to one, then either $\ell\cap \gamma_{j,i}$ is empty or consists of a unique point, and in the second case the crossing of $\gamma_{j,i}$ with $\ell$ is positive if and only if $\gamma_{j,i}$ is positive (and negative otherwise). By a connectivity argument, if $i_1<i_2$,  $\ell \cap \gamma_{j,i_1}\neq \varnothing$, $\ell \cap \gamma_{j,i_2}\neq \varnothing$, and $\ell \cap \gamma_{j,i}= \varnothing$ for all $i_1<i<i_2$, then $\gamma_{j,i_1}$ and $\gamma_{j,i_2}$ have different character. Therefore, the number of positive crossings minus the number of negative crossings of $\ell$ with $\gamma_j$ is either $0$ or $1$ or $-1$. 
It follows that $|w_\gamma(p)|$ is at most $2t_\gamma$, which concludes the proof.
\end{proof}

\section{Interpolation: proof of Theorem \ref{th:intro}}\label{sec:interpolation}

Polynomial Approximation is a fundamental subject in complex analysis, which is naturally linked in a strong way with Polynomial Interpolation. The classical Runge theorem says that for any compact set $K\subset\c$ such that $\c\setminus K$ is connected, every holomorphic function $f$ on a neighborhood of $K$ may be approximated, uniformly on $K$, by polynomials $\c\to\c$ (see \cite{Runge1885AM}). If in addition we are given a finite subset $\Lambda\subset K$, then the approximating polynomials can be chosen to agree with $f$ everywhere on $\Lambda$ (see Walsh \cite{Walsh1928TAMS}). Behnke and Stein extended Runge's theorem to the more general framework of meromorphic functions on compact Riemann surfaces (see \cite{BehnkeStein1949MA}); this was later generalized by ensuring, in addition, interpolation of arbitrary finite order on a finite set. To be precise, Royden proved that for any $\Sigma$, $E$, and $\Lambda$ as in Theorem \ref{th:intro}, given an integer $k\ge 1$ and a compact subset $K\subset\Sigma\setminus E$ such that $\Lambda\subset K$ and $E$ intersects all connected components of $\Sigma\setminus K$, then every holomorphic function $f$ on a neighborhood of $K$ may be approximated, uniformly on $K$, by meromorphic functions $F$ on $\Sigma$, being holomorphic in $\Sigma\setminus E$, such that the holomorphic function $F-f\colon K\to\c$ has a zero of order at least $k$ at all points in $\Lambda$ (see \cite[Theorem 10]{Royden1967JAM}).

We shall obtain Theorem \ref{th:intro} as a consequence of the following more precise result providing approximation, interpolation of given finite order, and control on the flux, which can be seen as an analogue for the above mentioned Royden theorem from  \cite{Royden1967JAM}. Recall that the flux map of a conformal minimal immersion $X\colon M\to\r^3$ of an open Riemann surface, $M$, is the group homomorphism $\Flux_X\colon \Hcal_1(M,\z)\to\r^3$, of the first homology group of $M$ with integer coefficients, given by 
\[
\text{$\Flux_X(\gamma)=\Im\int_\gamma \di X=-\imath\int_\gamma \di X$ for any loop $\gamma\subset M$.}
\]
\begin{theorem}\label{th:main-1}
Let $\Sigma$ be a compact Riemann surface with empty boundary, $\varnothing\neq E\subset\Sigma$ be a finite set, $K\subset \Sigma\setminus E$ be a compact  set such that every connected component of $\Sigma\setminus K$ contains points of $E$, and $\Lambda\subset K$ be a finite set. Also let $X\colon K\to\r^3$ be a conformal minimal immersion and $\pgot\colon \Hcal_1(\Sigma\setminus E,\z)\to\r^3$ be a group homomorphism such that
$\pgot(\gamma)=\Flux_X(\gamma)$ for all loops $\gamma\subset K$.
Then, for any integer $k\ge 0$ and any number $\epsilon>0$, there is a complete conformal minimal immersion $\wt X\colon\Sigma\setminus E\to\r^3$ satisfying the following conditions:
\begin{itemize}
\item[\rm (i)] $\wt X$ has finite total curvature.
\smallskip
\item[\rm (ii)] $|\wt X(\zeta)-X(\zeta)|<\epsilon$ for all $\zeta\in K$.
\smallskip
\item[\rm (iii)] $\wt X$ and $X$ have a contact of order $k$ at every point in $\Lambda$, i.e., $\wt X=X$ everywhere on $\Lambda$ and, if $k\ge 1$, $\di\wt X-\di X$ has a zero of order at least $k-1$ at every point in $\Lambda$.
\smallskip
\item[\rm (iv)] $\Flux_{\wt X}(\gamma)=\pgot(\gamma)$ for all loops $\gamma\subset \Sigma\setminus E$.
\end{itemize}
\end{theorem}

In case $\Lambda=\varnothing$ (i.e., without asking any interpolation property) the above theorem is already proved in \cite{Lopez2014TAMS}. On the other hand, he crucial novelty of Theorem \ref{th:main-1} is condition {\rm (i)}; except for it, the theorem  follows from the results in \cite{AlarconCastro2017}.   

\subsection{Proof of Theorem \ref{th:main-1}}
In view of \cite[Theorem 1.2 and Proposition 4.3]{AlarconCastro2017}, we may assume without loss of generality that $K$ is connected and a strong deformation retract of $\Sigma\setminus E$ and that $X$ is a nonflat conformal minimal immersion  with flux map $\pgot\colon \Hcal_1(\Sigma\setminus E,\z)\to\r^3$.

Call $(\phi_1,\phi_2,\phi_3)$ the Weierstrass data  and $(\eta_1,\eta_2)$ the spinorial representation of $X\colon K\to \r^3$. %Up to a rigid motion we can assume that $\phi_3$ does not take the values $\{0,\infty\}$ on $bK$, and so the spinorial representation $(\eta_1,\eta_2)$ associated to $X$ has no zeros on $bK$ as well. 
We need the following lemma. 

\begin{lemma}\label{lem:theta}
There exists a meromorphic $1$-form $\theta_0$ on $\Sigma$ meeting the following requirements:
	\begin{itemize}
		\item[{\rm a)}] $\theta_0$ is a spinorial holomorphic $1$-form on $\Sigma\setminus E$.
		\smallskip
		\item[{\rm b)}] $\theta_0$ does not vanish anywhere on $K$.
		\smallskip
		\item[{\rm c)}] $\theta_0$ is spinorially equivalent to $\eta_1$, and so to $\eta_2$, on $K$.
	\end{itemize}
\end{lemma}
\begin{proof}
%The existence of this $1$-form is guarantied by \cite[Claim 4.2]{Lopez2014TAMS}. 
Remark \ref{re:spinmero} provides a  meromorphic $1$-form $\omega$ on $\Sigma$ satisfying {\rm a)} and {\rm c)}. If $\omega$ meets {\rm b)} we are done, otherwise fix $q_0\in E$, write $(\omega|_{\Sigma\setminus E})=D_0^2$, and fix an open disk $W\subset\Sigma\setminus (K\cup E)$.
By \cite[Claim 4.1]{Lopez2014TAMS}, there is a meromorphic function $h_1\colon \Sigma\to\c$ such that 
$(h_1)=D_1 D_0^{-1} q_0^{-a}$ for some  integral divisor $D_1\in\Div(W)$  and  $a\in \z$.
%$h_1\neq 0$ on $K$ and $(h_1^2)\geq(\omega)^{-1}=D_0^{-2}$.
The $1$-form $\theta_0:=h_1^2\omega$ solves the lemma. Indeed, just observe that $(\theta_0)=D_1^2 D_E$, where $D_E\in \Div(E)$; in particular, $(\theta_0|_{\Sigma\setminus E})=D_1^2\geq 0$ and hence $\theta_0|_{\Sigma\setminus E}$ is holomorphic.
\end{proof}

Fix $\theta_0$ a $1$-form given by  Lemma \ref{lem:theta}. Let $u_j\in \Ocal(K)$ be a function satisfying
\begin{equation}\label{eq:uj}
u_j^2=\frac{\eta_j}{\theta_0}, \quad j=1,2.
\end{equation}
Since $X$ is nonflat, $u_j\colon K\to \c$ is non-identically zero and the divisor on $K$
\begin{equation}\label{eq:divw}
(u_j^2)=(\eta_j), \quad j=1,2.
\end{equation}

Fix $p_0\in \mathring K \setminus\Lambda$ such that $u_j(p_0)\neq 0$, $j=1,2$,  and write $\Lambda=\{p_1,\ldots,p_m\}$. Take $\{\gamma_1,\ldots,\gamma_l\}$, $l\geq m$,  curves in $K$ such that:
\begin{itemize}
\item $u_1$ and $u_2$ do not vanish anywhere  on   $\gamma_j$ for all $j=1,\ldots,l$,
\item $\gamma_j$ is a  smooth Jordan arc joining $p_0$ with $p_j$, $j=1,\ldots,m$, 
\item $\{\gamma_{m+1},\ldots,\gamma_l\}$ are loops in $K$ determining a basis of $\Hcal_1(K,\z)$, hence of $\Hcal_1(\Sigma\setminus E,\z)$, and 
\item $\gamma_i\cap\gamma_j=\{p_0\}$ for all $i\neq j \in \{1,\ldots,l\}$.
\end{itemize}
(Recall that $K$ is a strong deformation retract of $\Sigma\setminus E$.) 
Write $C=\bigcup_{j=1}^l\gamma_j$, and notice that $C\subset K$ is a Runge set in $\Sigma\setminus E$ which is a strong deformation retract of $\Sigma\setminus E$ as well.
Set 
\begin{equation}\label{eq:SpinK}
{\Spin}(K):= \{(f_1,f_2)\in \Ocal(K)^2\colon (|f_1|+ |f_2|)(p)\neq 0, \;\forall p\in K\}
\end{equation}
and observe that $(u_1,u_2)\in \Spin (K)$. Consider the continuous period map $\Pcal=(\Pcal_1,\ldots,\Pcal_l)\colon \Spin (K) \to\c^{3l}$ given by
\begin{equation}\label{eq:Periodofeta}
\Pcal_j(f_1,f_2)=\int_{\gamma_j}\big(\Phi(f_1,f_2)-\Phi(u_1,u_2)\big)\in\c^3 \quad \text{ for all $j=1,\ldots,l$}
\end{equation}
where $\Phi(\cdot,\cdot)$ is  formally defined by
\begin{equation}\label{eq:phioffunctions}
\Phi(f_1,f_2):=\bigg(\frac12(f_1^2-f_2^2)
,\frac\I2(f_1^2+f_2^2),
f_1f_2\bigg)\theta_0.
\end{equation}
Notice that $\Phi(u_1,u_2)=(\phi_1,\phi_2,\phi_3)$ are the Weierstrass data of $X$.

%\begin{equation}\label{eq:wtetas}
%\wt\eta_1:=\bigg(1+\sum_{j=1}^{l}\zeta_j h_j(p)\bigg)^2 \eta_1,
%\qquad
%\wt\eta_2:=\bigg(1-\sum_{j=1}^{l}\zeta_j h_j(p)\bigg)^2 \eta_2.
%\end{equation}

Let  $k\in\n$ be the integer number in the statement of the theorem. 

\begin{lemma}\label{lem:spray}
For any $i=1,2,3$,	there exist     $h^i_1,\ldots,h^i_{l}\in \Ocal(K)$ with a zero of multiplicity $2k\in\n$ at each point of $\Lambda$, such that:
\begin{itemize}
\item  The holomorphic map $\psi\colon \c^{3l} \to \Ocal(K)^2$ defined by
	\begin{equation}\label{eq:spray0}
	\psi(\zeta)= \big((1+\sum_{i=1}^{3}\sum_{j=1}^{l}\zeta^i_j h^i_j) u_1,  (1+2\sum_{i=1}^{3}\sum_{j=1}^{l}\zeta^i_j h^i_j) u_2 \big),
	\end{equation}
	where $\zeta=\big((\zeta_j^i)_{i=1,2,3}\big)_{j=1,\ldots,l}\in (\c^{3})^l$, assumes values in $ \Spin(K)$.
\item $\psi$ is a period dominating spinorial spray with core $(u_1,u_2)$; that is to say, $\psi(0)=(u_1,u_2)$ and the map	
	\[
			\Pcal \circ \psi\colon \c^{3 l}  \to \c^{3l}
	\]
 is a submersion at  $\zeta=0$. In particular, there exists a Euclidean ball $V_0\subset \c^{3 l}$ centered at the origin such that $(\Pcal \circ \psi)(V_0)\subset \c^{3 l}$ is a domain and  $\Pcal \circ \psi\colon V_0  \to (\Pcal \circ \psi)(V_0)$ is a biholomorphism.
\end{itemize}
\end{lemma} 
\begin{proof}
Since $X$ is a nonflat conformal minimal immersion, we may choose for any $j=1,\ldots,l$ pairwise distinct points $p_j^i\in\gamma_j$, $i=1,2,3$ (different from the endpoints) such that the vectors 
\begin{equation}\label{eq:basis}
%\big\lbrace\Phi(u_1,u_2)(p_j^i)\big\rbrace_{i=1,2,3}=
\left\lbrace \bigg(\frac12(u_1^2-u_2^2)
,\frac\I2(u_1^2+u_2^2),
u_1u_2\bigg)(p_j^i)\right\rbrace _{i=1,2,3}
\end{equation}
are a basis of $\c^3$.
Next, for any $i=1,2,3$, we consider continuous functions $h^i_j\colon C\to\c$ such that $h^i_j$ vanishes on $C\setminus \gamma_j$; the value of each $h^i_j$ in $\gamma_j$ will be specified later.
Notice that, so far,  we have defined the functions $h_j^i$'s only in $C$. In particular, the expression $\Pcal\circ \psi$ makes sense in a natural way; see \eqref{eq:Periodofeta}.

Thus, the differential of $\Pcal\circ\psi=(\Pcal_1\circ\psi,\ldots,\Pcal_l\circ\psi)$ with respect to $\zeta^i_j$ is given for any $i=1,2,3$ and $j=1,\ldots,l$ by
\[
	\frac{\partial \Pcal_{m}\circ\psi}{\partial \zeta^i_j}\bigg|_{\zeta=0}(\zeta)=\left\lbrace 
	\begin{matrix}
	(0,0,0) & j\neq m\\
	\\
	\displaystyle\int_{\gamma_j}h^i_j\big(u_1^2-2u_2^2,\I (u_1^2+2u_2^2),3u_1u_2\big)\, \theta_0 & j=m
	\end{matrix}\right. 
\]
for any $m=1,\ldots,l$.

We claim that we may choose the functions $h_j^i$ so that the vectors $\frac{\partial \Pcal_j\circ\psi}{\partial \zeta^i_j}\big|_{\zeta=0}$, $i=1,2,3$, expand $\c^3$, and so the differential of $\psi$ at $\zeta=0$ is surjective. Indeed, we parametrize each curve by $\gamma_j\colon [0,1]\to \gamma_j\subset K$ (we identify the image $\gamma_j([0,1])\equiv \gamma_j$) and call $t_j^i\in(0,1)$ the point such that $\gamma_j(t_j^i)=p_j^i$ for $i=1,2,3$. Then, take a positive number $\tau>0$ small enough such that 
\[
	[t_j^{i}-\tau,t_j^{i}+\tau]\subset (0,1)\quad \text{and} \quad	[t_j^{i_1}-\tau,t_j^{i_1}+\tau]\cap[t_j^{i_2}-\tau,t_j^{i_2}+\tau]=\varnothing
\]
for any $i,i_1\neq i_2\in\{1,2,3\}$
and define the function $h_j^i$ such that 
\begin{equation}\label{eq:joaq}
h_j^i=0 \quad \text{everywhere on $[0,1]\setminus [t_j^{i}-\tau,t_j^{i}+\tau]$}
\end{equation}
and
\[
\int_{0}^1h^i_j(t) \,dt = \int_{t^i_j-\tau}^{t^i_j+\tau}h^i_j(t) \,dt=1.
\]
Therefore, we may ensure that for sufficiently small $\tau>0$, we have that the integral
\[
\int_{\gamma_j}h^i_j\big(u_1^2-2u_2^2,\I (u_1^2+2u_2^2),3u_1u_2\big)\, \theta_0
\]
takes approximately the value
%\[
%\bigg(\frac12(u_1-u_2)
%,\frac\I2(u_1+u_2),
%\sqrt{u_1u_2}\bigg)(\gamma(t_j^i)) \ \theta(\gamma(t_j^i),\dot\gamma(t_j^i)).
%\]
\[
\big((u_1^2-2u_2^2)
,\I(u_1^2+2u_2^2),
3u_1u_2\big)(\gamma_j(t_j^i)) \ \theta_0(\gamma_j(t_j^i),\dot\gamma_j(t_j^i)).
\]
Finally, since $\theta_0$ does not vanish anywhere on $C\subset K$ and the vectors in equation \eqref{eq:basis} are a basis of $\c^3$, then the vectors $\left\lbrace \big(u_1^2-2 u_2^2,\I(u_1^2+2 u_2^2),3u_1 u_2\big)(p_j^i)\right\rbrace_{i=1,2,3}$ expand $\c^3$ for all $j=1,\ldots,l$. Indeed, notice that
\[
	\big(\frac12(u_1^2-u_2^2)
	,\frac\I2(u_1^2+u_2^2),
	u_1u_2\big)
	\left( \begin{matrix}
	3 & -\I & 0\\
	\I & 3 & 0\\
	0 & 0 & 3\\
	\end{matrix}\right)= 
	\big(u_1^2-2 u_2^2,\I(u_1^2+2 u_2^2),3u_1 u_2\big).
\]
This shows that $\Pcal\circ\psi$ is a submersion at $\zeta=0$.

To finish the proof we claim that we may assume that each $h_j^i$ is a holomorphic function $h_j^i\colon K\to\c$, vanishes at all zeros of $u_1 u_2$, and  has a zero of multiplicity at least $2k$ at any point of $\Lambda$. 
Indeed, by Mergelyan Theorem with jet-interpolation (see \cite{Forstneric2017}) we may approximate each $h_j^i$ by a holomorphic function $K\to\c$. Since each $h_j^i$ vanishes on a neighborhood of $\Lambda$ in $C$ (see \eqref{eq:joaq}), then the approximating function may be chosen {\em to vanish at all zeros of $u_1 u_2$} and to have a zero of any order at any point of $\Lambda$ (in particular of order $2k$). The former and the fact that $(u_1,u_2)\in \Spin (K)$ ensure that the map $\psi$ defined in \eqref{eq:spray0} assumes values in $\Spin (K)$. Furthermore, if the approximation is close enough then the corresponding spray is also period dominating. 
\end{proof}

We now prove the following generalization of Royden's theorem.

\begin{lemma} \label{le:polos}
Let $A\subset \Sigma \setminus E$ be a Runge compact subset, let $f\in \Mgot(A)$ be a meromorphic function,  choose $m \in \n$, take an integral divisor $D\in \Div(A)$,  and fix $\delta>0$.  Then there exists $\wt f\in \Mgot(\Sigma)\cap  \Ocal(\Sigma \setminus (A\cup E))$ such that $\wt f$ has a pole of order greater than or equal to $m$ at all points in  $E$, $(\wt f|_A-f)\geq D$,  and $|\wt f-f|<\delta$ on $A$.
\end{lemma}
\begin{proof} Without loss of generality   suppose that $\delta<1$. Denote by $r$  the number of points in $E$. By Royden's Theorem  \cite[Theorem 10]{Royden1967JAM}, for each $q\in E$  we can find $f_q\in \Mgot(\Sigma)\cap \Ocal(\Sigma\setminus \{q\})$ such that $|f_q|<\delta/2r$ on $A$,  $(f_q)\geq D$, and having a pole of positive order at $q$. Call $f_0=\sum_{q\in E} f_q$, and observe that $|f_0|<\delta/2$ on $A$,  $(f_0|_A)\geq D$, and $f_0$ has  a pole of positive order at all points of  $E$. Likewise, Royden's theorem provides $f_1\in \Mgot(\Sigma)\cap \Ocal(\Sigma\setminus (A\cup E))$ such that $|f-f_1|<\delta/2$ and  $(f_1|_A-f)\geq D$. Label by $m_0\in \n$ the maximum order of the poles of $f_1$ at points of $E$. It suffices to set $\wt f:=f_1+ f_0^n$, where $n=m_0+m$.
\end{proof}

Fix $\rho>0$ to be specified later. 

Choose pairwise disjoint compact discs $U_q\subset \Sigma\setminus (K\cup E)$, $q\in\supp(\theta_0|_{\Sigma\setminus E})$, with $q\in\mathring U_q$; recall that $\theta_0$ vanishes nowhere on $K$, see Lemma \ref{lem:theta}. Set $K_0=\bigcup_{q\in \supp(\theta_0)_0} U_q$ and extend $u_1$ to  a    meromorphic function $u_1\colon K\cup K_0\to \c$ such that  $u_1^2 \theta_0$ is holomorphic and vanishes nowhere on $K_0$ (recall that $\theta_0$ is spinorial on $\Sigma \setminus E$).  Fix another number $\rho'>0$ which will be specified later. If $\rho'>0$ is sufficiently small, then Lemma \ref{le:polos} applied to the Runge set $K\cup K_0$, $u_1$, a large enough integer $m\in \n$, $D=(u_1)^2 D_{\Lambda}^{2 k}$, and $\rho'>0$, where $D_\Lambda=\prod_{q\in \Lambda} q\in \Div(K)$, provides a meromorphic function $\wt u_1\colon\Sigma\to\c$,  holomorphic on $\Sigma\setminus (E\cup \supp(\theta_0)_0)$, satisfying the following properties:
\begin{enumerate}[\rm ({A}.1)]
	\item $(\wt u_1|_K)=(u_1|_K)$,
	\smallskip
	\item $\wt u_1-u_1$ has a zero of multiplicity at least  $2k$ at all points of $\Lambda$,
	\smallskip
	\item $|\wt u_1(\zeta)-u_1(\zeta)|<\rho'$ for all $\zeta\in K\cup K_0$, and
	\smallskip
	\item $\wt u_1^2 \theta_0$ is holomorphic on $\Sigma\setminus E$, vanishes nowhere on $K_0$, and has a pole of positive order at all points of $E$.
\end{enumerate}
Notice that in order to ensure {\rm (A.1)} and {\rm (A.2)} we are using that $\rho'>0$ is small enough, and taking into account classical Hurwitz's theorem.  Also, to guarantee {\rm (A.4)} we are using {\rm (A.3)} and the definition of $u_1$ on $K_0$, and assuming that $m$ is large enough.

Consider the finite set $K_1:=\{p\in \Sigma\setminus K \colon \wt u_1(p)=0\}$ and observe that $K_1\cap E= \varnothing$ (take into account that $\wt u_1$ has poles of order at least $m\geq 1$ at all points of $E$), hence $K_1\subset \Sigma\setminus (K\cup K_0\cup E)$. Extend $u_2$ to a meromorphic function  $u_2\colon K\cup K_0\cup K_1\to \c$  satisfying $u_2=u_1$ on $K_0$ and $u_2=1$ on $K_1$.
As above,   Lemma \ref{le:polos} provides a meromorphic function $\wt u_2\colon\Sigma\to\c$,  holomorphic on $\Sigma\setminus (E\cup \supp (\theta_0)_0)$, satisfying the following properties:
\begin{enumerate}[\rm ({B}.1)]
	\item $(\wt u_2|_K)=(u_2|_K)$,
	\smallskip
	\item $\wt u_2-u_2$ has a zero of multiplicity at least $2k$ at all points of $\Lambda$, 
	\smallskip
	\item  $|\wt u_2(\zeta)-u_2(\zeta)|<\rho'$ for all $\zeta\in K\cup K_0\cup K_1 $, and
	\smallskip
	\item $\wt u_2^2 \theta_0$ is holomorphic on $\Sigma\setminus E$, vanishes nowhere on $K_0\cup K_1$, and has an effective pole at all points of $E$.
\end{enumerate}
By  {\rm (A.1)}, {\rm (A.4)}, {\rm (B.1)}, {\rm (B.3)}, and {\rm (B.4)}, 
\begin{equation}\label{eq:B3}
\text{$(\wt u_1^2,\wt u_2^2)\theta_0$ is holomorphic on $\Sigma\setminus E$ and does not assume the value $(0,0)$},
\end{equation}
provided that $\rho'>0$ is sufficiently small.
Next, set 
\begin{equation}\label{eq:wteta}
\wt \eta_j:=\wt u_j^2 \theta_0,\quad j=1,2.
\end{equation}

Summarizing:
\begin{enumerate}[\rm (a)]
	\item $\wt \eta_j$ is holomorphic and spinorial on $\Sigma\setminus E$, $j=1,2$.
	\smallskip
	\item The pair $(\wt \eta_1,\wt \eta_2)$ is spinorially equivalent on $\Sigma\setminus E$; in fact, $\wt \eta_j$ is spinorially equivalent to $\eta_1$ (and $\eta_2$) on $K$, $j=1,2$ (recall that $K$ is a strong deformation retract of $\Sigma\setminus E$).
	\smallskip
	\item $\wt \eta_j-\eta_j$ has a zero of multiplicity at least $2k$ at all points of $\Lambda$.
	\smallskip
	\item $\wt \eta_j$ has an effective pole at any point of $E$, $j=1,2$; take into account ({\rm A}.4) and ({\rm B}.4).
	\smallskip
	\item $|(\wt \eta_j-\eta_j)/\theta_0|<\rho$  everywhere on $K$. For that we use {\rm (A.3)} and {\rm (B.3)} and assume that $\rho'>0$ is chosen sufficiently small.
	\smallskip
	\item $(\wt \eta_j|_K)=(\wt u_j^2|_K)=(u_j^2|_K)=(\eta_j)$, $j=1,2$, and $\wt \eta_1$ and $\wt \eta_2$ have no common zeros on $\Sigma\setminus  E$.
\end{enumerate}

%\begin{equation}\label{eq:divf}
%(\wt f_j^2|_K)=(f_j^2)=(\eta_j), \quad j=1,2.
%\end{equation}

Let $h_j^i\colon K\to\c$, $j=1,\ldots,l$, $i=1,2,3$, be the holomorphic functions given by Lemma \ref{lem:spray}.
Applying Lemma \ref{le:polos} once again to $A=K$, $m=2k$, $D=\prod_{p\in \Lambda} p^{2 k}$, and $\delta =\rho$,  we get functions $g_j^i\in \Mgot(\Sigma)\cap \Ocal(\Sigma\setminus E)$,  $j=1,\ldots,l$, $i=1,2,3$, meeting the following properties:
\begin{enumerate}[\rm (I)]
	\item $|g_j^i(\zeta)-h_j^i(\zeta)|<\rho$ for all $\zeta\in K$.
	\smallskip
	\item $g_j^i-h_j^i$ has a zero of multiplicity at least $2k\in\n$ at any point $p\in\Lambda\subset K$. This is equivalent to that $g_j^i$ has a zero of multiplicity at least $2k\in\n$ at any point $p\in\Lambda\subset K$ (see Lemma \ref{lem:spray}).
	\smallskip
	\item $g_j^i$ has a zero at any point of $(\Sigma\setminus E)\cap \big(\supp (\wt \eta_1) \cup \supp (\wt \eta_2)\big)$.
	\item $g_j^i$ has a pole of order at least $2k$ at any point of $E$ for all $j$ and $i$.
	\end{enumerate}
 
Replacing  in \eqref{eq:spray0} the functions $h_j^i$, $u_1$, and $u_2$ by their corresponding approximations, namely, $g_j^i$, $\wt u_1$, and $\wt u_2$, we get the holomorphic spinorial spray on $K$ 
 \[
\wt\psi=(\wt\psi_1 ,\wt\psi_2 )\colon \c^ {3 l}\to \big(\Mgot(\Sigma)\cap \Ocal(\Sigma\setminus E)\big)^2\cap \Spin(K),
\]
given by 
\[
		\wt\psi(\zeta)= \bigg((1+\sum_{i=1}^{3}\sum_{j=1}^{l}\zeta^i_j g^i_j) \wt u_1,  (1+2\sum_{i=1}^{3}\sum_{j=1}^{l}\zeta^i_j g^i_j) \wt u_2 \bigg), \quad \zeta\in \c^{3 l},
\]
with  core $(\wt u_1,\wt u_2)$.
Set  
\begin{equation}\label{eq:spraypsi}
	({\wt \eta}_{\zeta,1},{\wt \eta}_{\zeta,2}):=
	\bigg((1+\sum_{i=1}^{3}\sum_{j=1}^{l}\zeta^i_j g^i_j)^2 
	{\wt \eta}_1, (1+2\sum_{i=1}^{3}\sum_{j=1}^{l}\zeta^i_j g^i_j)^2 {\wt \eta}_2 \bigg)
\end{equation}
and notice that ${\wt \eta}_{\zeta,j}=\wt\psi_j(\zeta)^2\theta_0$, $j=1,2$; see \eqref{eq:wteta}. 
Moreover: 
\begin{enumerate}[{\rm (A)}]
\item   ${\wt \eta}_{\zeta,j}$ is a spinorial holomorphic $1$-form on $\Sigma\setminus E$, i.e., ${\wt \eta}_{\zeta,j}\in \Ygot(\Sigma\setminus E)$, $j=1,2$. Furthermore, ${\wt \eta}_{\zeta,1}$ and  ${\wt \eta}_{\zeta,2}$ are meromorphic on $\Sigma$ and spinorially equivalent on $\Sigma\setminus E$ and have no common zeros on $\Sigma\setminus E$, and either ${\wt \eta}_{\zeta,1}$ or ${\wt \eta}_{\zeta,2}$ has an effective pole at any point of $E$; use properties {\rm (a)}, {\rm (b)}, {\rm (d)}, {\rm (f)},  {\rm (III)},  and equation  \eqref{eq:spraypsi}. Take into account that, obviously, $1+\sum_{i=1}^{3}\sum_{j=1}^{l}\zeta^i_j g^i_j$ and $1+2\sum_{i=1}^{3}\sum_{j=1}^{l}\zeta^i_j g^i_j$ have no common zeros.
\smallskip

\item If $\rho>0$ is chosen small enough,  the period map $\wt \Psf \colon \c^{3 l}\to \c^{3l}$, $\wt \Psf(\zeta)  =  \Pcal\big(\wt \psi(\zeta)\big)$, see \eqref{eq:Periodofeta}, is period dominating at $\zeta=0$ in the sense that 
$\wt \Psf $  is a submersion at that point; take into account Lemma \ref{lem:spray}. Therefore, there is an Euclidean ball $V\subset \c^{3 l}$ centered at the origin, and depending on $\rho$, such that $\wt \Psf \colon V\to \wt \Psf(V)$ is a biholomorphism; see  {\rm (e)}, \eqref{eq:wteta}, and {\rm (I)}.
\smallskip

\item Since $\wt \Psf$ approximates $\Pcal\circ \psi$ uniformly on compact subsets of $\c^{3 l}$ as $\rho$ goes to $0$, $\psi$ is period dominating, and $\Pcal(\psi(0))=0$, then there is $\zeta_\rho\in V$, which goes to $0$ as $\rho\to 0$, such that $\wt \Psf(\zeta_\rho)=0$; take into account {\rm (B)} and the definition of $\wt \Psf$.
\smallskip

\item  ${\wt \eta}_{\zeta_\rho,j}/\theta_0$ approximates $\eta_j/\theta_0$ uniformly on $K$ as $\rho$ goes to $0$. Take into account {\rm (e)} and the fact that $\lim_{\rho\to 0} \zeta_\rho=0$; see {\rm (C)}.
\smallskip

\item ${\wt \eta}_{\zeta,j} -\eta_j$ has a zero of multiplicity at least $2 k$ at any point of $\Lambda$ for all $\zeta\in \c^{3l}$; see {\rm (c)} and {\rm (II)}.
\end{enumerate}

By properties {\rm (A)} and {\rm (C)}, the Weierstrass data $\wt\Phi_{\zeta_\rho}:=\Phi(\wt\psi(\zeta_\rho))$ defined as in equation \eqref{eq:phioffunctions} provides a complete conformal minimal immersion 
\[%\begin{equation}\label{eq:wtX}
	X_\rho(p)=X(p_0)+\Re\int_{p_0}^{p}\wt \Phi_{\zeta_\rho}.
\]%\end{equation}
Note that $X_\rho$ is well defined since the periods of $\wt\Phi_{\zeta_\rho}$ on $K$ are those of the Weierstrass data of $X$ and $K$ is a strong deformation retract of $\Sigma\setminus E$.
The completeness is ensured by the fact that either ${\wt \eta}_{\zeta,1}$ or ${\wt \eta}_{\zeta,2}$ has an effective pole at any point of $E$.

We claim that the immersion $\wt X:=X_\rho$ solves the theorem provided that $\rho>0$ is small enough. 
Indeed, by the aforementioned argument, $X_\rho$ has the same flux map as $X$; this gives  Theorem \ref{th:main-1}-{\rm (iv)}.  By property {\rm (A)}, the immersion $X_\rho$ is of finite total curvature, proving condition  Theorem \ref{th:main-1}-{\rm (i)}. Condition  Theorem \ref{th:main-1}-{\rm (ii)} follows from {\rm (D)} for $\rho$ small enough. Finally, if $p_j\in \Lambda$ and $\gamma_j$ is the arc in $C$ joining $p_0$ and $p_j$, then 
\begin{eqnarray*}
	X_\rho(p_j) & = & X(p_0)+\Re\int_{p_0}^{p_j}\wt \Phi_{\zeta_\rho}
	\; = \; X(p_0)+ \Re\int_{\gamma_j} \wt \Phi_{\zeta_\rho} \; \stackrel{{\rm (C)}}{=}
	 \\
	 & \stackrel{{\rm (C)}}{=} & X(p_0)+ \Re\int_{\gamma_j} \Phi(u_1,u_2)
	 \; = \; X(p_0)+ \Re\int_{p_0}^{p_j} \Phi(u_1,u_2) \; =\; X(p_j).
\end{eqnarray*}
Together with  {\rm (E)} we infer Theorem \ref{th:main-1}-{\rm (iii)}.

This completes the proof of Theorem \ref{th:main-1}.
Theorem \ref{th:intro} is proved.

%%%%%%%%%%
%%%%%%%%%%
%%%%%%%%%%
%%%%%%%%%%   Section: Optimal hitting
%%%%%%%%%%
%%%%%%%%%%

\section{Optimal hitting: proof of Theorem \ref{th:main-intro-2}}\label{sec:hitting}

Let $X\colon M\to \r^3$ be a complete  conformal minimal  immersion of FTC, where $M$ is an open Riemann surface. We know that $M=\Sigma\setminus E$, where $\Sigma$ is a compact Riemann surface and $E\subset \Sigma$ is  a non-empty finite subset. Denote  by   $N\colon M\to \s^2$ the Gauss map of $X$ which is compatible with the orientation on $M$,  and recall that $N$ extends conformally to $\Sigma$. 

%As a first result we prove the following proposition.
%\begin{proposition}\label{pro:4points}
%	Let $A\subset \r^3$ be a subset consisting on 4 point not contained in a plane. Then there exists an Enneper's surface which contains $A$.
%\end{proposition}
%\begin{proof}
%	Enneper's surface is a simply-connected and self-intersecting minimal surface of FTC equals to $-4\pi$. It can be parametrized as $X\colon \c\to\r^3$ by
%	\[
%	X(z)\equiv X(u+\I v):=\left( \frac u3 \big(1 - \frac{u^2}3 + v^2\big), -\frac{v}3 \big(1 - \frac{v^2}3 + u^2\big),  \frac{u^2 - v^2}3\right)
%	\]
%	for all $z=u+\I v\in\c$.
%	
%	It is clear that $X(\c)$ contains the lines
%	\begin{equation}\label{eq:lines}
%	\left.\begin{matrix}
%	x_1+x_2=0\\
%	x_3=0
%	\end{matrix}\right\rbrace
%	\text{ and }
%	\left.\begin{matrix}
%	x_1-x_2=0\\
%	x_3=0
%	\end{matrix}\right\rbrace
%	\end{equation}
%	
%	Up to a rigid motion suppose that $3$ of the points of $A$ lies in the lines described in equation \eqref{eq:lines}. Obviously, any dilation of the Enneper surface $X$ still contains those lines. Since dilations of the Enneper surface $X$ are another Enneper surface. We may find an Enneper surface that assumes the values of $A$.
%\end{proof}	

Before going into the proof of the theorem some preparations are needed.

By definition, a symmetry of $X\colon M=\Sigma\setminus E\to\r^3$ is a rigid motion  $\Rcal\colon \r^3\to \r^3$ such that $\Rcal(X(M))=X(M)$. Every symmetry $\Rcal$ of $X$ induces a conformal transformation $\Psi \colon M\to M$ satisfying $\Rcal \circ X=X\circ \Psi $.   Such a map $\Psi $ could be either orientation  preserving or   orientation reversing, and extends to a conformal automorphism $\Psi_\Rcal\colon \Sigma\to \Sigma$ leaving invariant  $E$.
By analyticity, any affine line $L\subset\r^3$ for which  $X(M)\cap L$ consists of infinitely many points is contained in $X(M)$. If $L$ is an affine line and $L\subset X(M)$, then Schwartz's reflection principle implies that  $X(M)$ is invariant under the reflection $\Rcal_L\colon\r^3\to\r^3$ about $L$. Furthermore, $\Rcal_L$ has an associated conformal transformation, namely  $\Psi_L$, with infinitely many fixed points on $\Sigma$ (hence antiholomorphic). 

\begin{proposition}\label{pro:isomet}
The only complete orientable minimal surfaces of finite total curvature with an infinite group of symmetries are the planes and the catenoids. In particular,   every complete {\em non-flat}  minimal surface of finite total curvature contains at most a finite number of straight lines.
\end{proposition}  
We think that Proposition \ref{pro:isomet} is well-known, we include here a proof since we have been unable to find a precise reference in the literature.
\begin{proof} 
Let $X\colon  M=\Sigma\setminus E\to \r^3$ be a complete non-flat  conformal minimal  immersion of FTC,  where $\Sigma$ is a compact Riemann surface and $E\subset \Sigma$ is a non-empty finite subset. Since the Gauss curvature $K\colon M\to \r$ is non-positive and $K(p)\to 0$ as $p\to E$, the set  $C=\{p\in M\colon K(p)=\min_\Sigma K<0\}$ is non-empty and compact; recall that $K$ does not vanish everywhere on $M$ since $X$ is nonflat.  Call $G$ the symmetry group of $X$ and notice that $G$  leaves $C$ invariant. Furthermore, $G$ is a closed subgroup of the Lie group of rigid motions in $\r^3$.

If we assume that   $G$ is not discrete, then $G$ contains a 1-parametric subgroup, $G_0$,  leaving  invariant the compact set $X(C)$. Therefore, $G_0$ consists of rotations, and so $X(M)$ is the catenoid. This concludes the proof under this assumption. 

To finish it remains to show that $G$ is not discrete. Indeed, assume that $G$ is discrete and recall that, by assumption, $G$ is infinite. Fix an end $p\in E$, and denote by $G_p$ the subgroup of $G$ consisting of those isometries inducing a conformal transformation in $\Sigma$ that fixes $p$.  Clearly $G_p\neq \varnothing$ is discrete and infinite as well; for the latter recall that $E$ is finite and every symmetry in $G$ leaves $E$ invariant.  As above, $G_p$ leaves $X(C)$ invariant, hence its associated group of linear isometries $\vec G_p$  is discrete as well. Denote by  $N\colon \Sigma\to \s^2$ the extended Gauss map of $X$, and observe that every isometry in $\vec G_p$ leaves invariant the vectorial direction $L_p$ generated by $N(p)$. We infer that $\vec G_p$ must be a finite group of linear isometries leaving   $L_p$ invariant, and so that  $G_p$ must contain either a screw motion or a sliding symmetry. This contradicts that $C$ is invariant under $G_p$.
\end{proof}

%
% Proof
%

%We now prove the main result in this section.

\begin{proof}[Proof of Theorem \ref{th:against}]  
If $X$ is flat then the conclusion of the theorem trivially holds. 

Assume that $X$ is nonflat. As above we put $M=\Sigma\setminus E$, where $\Sigma$ is a compact Riemann surface and $E\subset \Sigma$ is  a non-empty finite subset, and call   $N\colon \Sigma \to \s^2$ the extended Gauss map of $X$ compatible with the orientation on $M$. Since $L\not\subset X(M)$ we have that $X^{-1}(L)\subset M$ is a finite subset, by analyticity.

Denote by $\Gcal$ the Grassmanian manifold of all affine lines in $\r^3$. By Proposition \ref{pro:isomet}, $X(M)$ contains at most finitely many lines in $\Gcal$. Denote by $\Gcal_0$ the subset of $\Gcal$ consisting of those affine lines $T$ in $\r^3$ which satisfy the following properties:
\begin{itemize}
 \item[{\rm a)}] $Q_T:=\{p\in M\colon N(p)\bot T\}$ is compact, and hence $T\not\subset X(M)$. 
 \smallskip
\item[{\rm b)}] The Gauss curvature $K\colon M\to \r$ of $X$ vanishes nowhere on  $Q_T$, that is to say, $N$ is a local diffeomorphism around points of $Q_T$.
\smallskip
\item[{\rm c)}]  $X^{-1}(T)\cap Q_T=\varnothing$.
\end{itemize}
Clearly $\Gcal_0$   is open and dense in $\Gcal$.  Concerning   {\rm c)}, notice that $Q_T$ only depends on the direction vector of $T$.

Let us show the following reduction.
%
% Claim
%
\begin{claim}\label{cl:abc}
It suffices to prove the theorem under the extra assumption that $L\in\Gcal_0$.
\end{claim}

\begin{proof} 
Suppose for a moment that under the assumptions  {\rm a)}, {\rm b)}, and {\rm c)}  the conclusion of the theorem holds. By this assumption, 
\begin{equation}\label{eq:T}
	\text{$\# \big(X^{-1}(T)\big)\leq 6 {\rm Deg}(N)+\chi(M)$ holds for every $T\in \Gcal_0$.}
\end{equation}
Let us show that the same inequality occurs for an arbitrary $L\in\Gcal$ with $L\not\subset X(M)$. Indeed, choose pairwise disjoint compact neighborhoods $U_p$ of each $p\in X^{-1}(L)\subset M$ in $M$, and  notice that $\Gcal_1=\{T\in \Gcal_0\colon T\cap X(U_p)\neq \varnothing \, \;\forall p\in  X^{-1}(L)\}$ is a non-empty subset of  $\Gcal$ whose closure contains $L$. Since $\# \big(X^{-1}(L)\big)\le \# \big(X^{-1}(T)\big)$ for all $T\in \Gcal_1$, \eqref{eq:T} implies that $\# \big(X^{-1}(L)\big)\leq 6 {\rm Deg}(N)+\chi(M)$.
\end{proof}

To complete the proof it therefore remains to prove the theorem assuming that $L\in\Gcal_0$. We proceed with that.

Up to a rigid motion, we may suppose that $L\in\Gcal_0$ is the $x_3$-axis.
Write $X=(X_j)_{j=1,2,3}$ and call $\vec e_3=(0,0,1)$. It follows from {\rm a)}, {\rm b)}, and {\rm c)} that   $\langle N,\vec e_3\rangle\neq 0$ everywhere on $X^{-1}(L)$, $0$ is a regular value of $\langle N,\vec e_3\rangle\colon M\to\r$, and $Q:= Q_L= \langle N,\vec e_3\rangle^{-1}(0)$ is a compact subset in $M=\Sigma \setminus E$  consisting of finitely many pairwise disjoint regular analytical Jordan curves. We are using for the last assertion that the degree of $N$ is finite.

\begin{remark} \label{re:tv}
Notice that the tangent vector of any  curve  in $X(Q)$ may be parallel to $(0,0,1)$ at some points, and so the projection of the curve into any horizontal plane could contain cusp points (see Subsec.\ \ref{ss:cusp}). Indeed, by solving suitable Bj\"orling problems one can easily construct minimal surfaces in $\r^3$ containing closed regular analytical curves  with horizontal Gauss map and vertical tangent vector at some points (see Figure \ref{fig:basis}).  
 \begin{figure}[ht]
    \begin{center}
 \scalebox{0.3}{\includegraphics{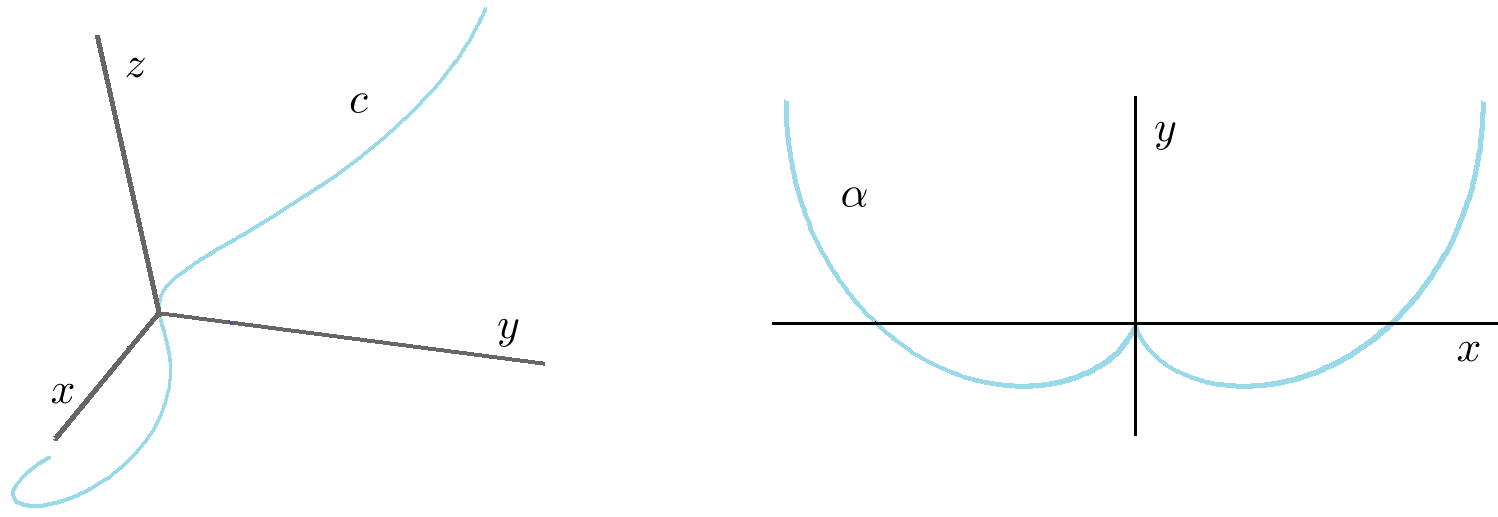}}
     \end{center}
\caption{A regular analytic curve $c$ in $\r^3=\r^2\times \r$ whose projection  $\alpha\subset \r^2$ produces a cusp point at the origin.}
\label{fig:basis}
\end{figure}
However, by analyticity, the tangent vector in $X(Q)$ is parallel to $(0,0,1)$ at most on finitely many points; otherwise    the curve would be a vertical straight line, hence non-compact. 
\end{remark}

Let $\Omega$ be a connected component of $\Sigma\setminus Q$. Obviously $b\Omega:=\overline\Omega\setminus\Omega\subset Q$  contains no point of $X^{-1}(L)\cup E$, by properties {\rm a)} and {\rm c)}.
Write $X^{-1}(L)\cap \overline \Omega=\{p_1,\ldots,p_s\}\subset    \Omega\setminus  E = \overline \Omega\setminus (Q\cup E)$ and  choose  pairwise disjoint   compact discs $D_1,\ldots, D_s$  in $  \Omega\setminus  E$ satisfying the following requirements:
\begin{itemize}
\item  $D_j$ contains $p_j$ as interior point. 
\item $(X_1,X_2)|_{D_j}$ is one to one; take into account that the Gauss map of $X|_{D_j}$  is never horizontal and choose $D_j$ small enough.
\end{itemize}

Likewise, write $ E\cap   \overline \Omega =\{q_1,\ldots,q_r\}\subset \Omega$ (possibly  $E\cap   \overline \Omega=\varnothing$ and $r=0$) and as above take pairwise disjoint   compact discs $U_1,\ldots, U_r$  in $\Omega$ such that  
\begin{itemize}
\item  $U_j$ contains $q_j$ as interior point and $U_j\setminus \{q_j\}\subset \Omega\setminus   \bigcup_{j=1}^s D_j $. 
\item $(X_1,X_2)|_{U_j}$ is an  $I_{q_j}$-sheeted multigraph   over $\r^2\setminus B$, where $I_{q_j}\geq 1$ and $B$ is  an open Euclidean disc $B$  not depending on $j$.
\end{itemize}
Such discs exist by the well-known asymptotic behavior of the ends of complete minimal surfaces with finite total curvature in $\r^3$; see Subsec.\ \ref{ss:FTC} and Jorge and Meeks \cite{JorgeMeeks1983T}.

Let $c_1,\ldots,c_m$ denote the family of pairwise disjoint Jordan curves in $Q\cap \overline \Omega = b \Omega$. 
Denote by $d_j=bD_j :=D_j\setminus \mathring D_j$,  $j=1,\ldots,s$,   $u_j=bU_j:=U_j\setminus \mathring U_j$, $j=1,\ldots,r$,  where the symbol $\mathring\;$ denotes topological interior in $\Sigma$. Consider the compact region in $\Sigma$
\[
	\Omega_0:=\overline \Omega\setminus \big((\bigcup_{j=1}^s \mathring D_j) \cup (\bigcup_{j=1}^r \mathring U_j) \big),
\]
and observe that $\Omega_0\subset M$.
Introduce the continuous map 
\[
f\colon \Omega_0\to \s^1, \quad f(p):= \frac{(X_1,X_2)}{\|(X_1,X_2)\|},
\]
and consider the induced group morphism between the first homology groups
\[
f_*\colon \Hcal_1(\Omega_0,\z)\to \Hcal_1(\s^1,\z)\equiv\z.
\]
Note that
\[
	b\Omega_0=(\bigcup_{j=1}^m c_j)  \cup (\bigcup_{j=1}^s d_j) \cup (\bigcup_{j=1}^r u_j)
\]
and endow the curves $c_1,\ldots,c_m$, $d_1,\ldots,d_s$,  $u_1,\ldots, u_r$ with the {\em orientation induced by the region $\Omega_0\subset M$}. It follows that  $(\sum_{j=1}^m c_j)+(\sum_{j=1}^s d_j)+(\sum_{j=1}^r u_j)=0$ in $ \Hcal_1(\Omega_0,\z)$,  and so
\begin{equation}\label{eq:gradocur}
\sum_{j=1}^s f_*(d_j)=- \sum_{j=1}^r f_*(u_j)-\sum_{j=1}^m f_*(c_j)\in\z.
\end{equation}

 Since $(X_1,X_2)|_{D_j}$ is one to one, the winding number  of $f(d_j)$ with respect to the origin  is equal to $\pm 1$. Further, since $\bigcup_{j=1}^s D_j\subset\Omega$ and the Gauss map of $X|_\Omega$ assumes values in a hemisphere, the sign depends on the fixed orientation in $\Omega_0$ but not on $j\in \{1,\ldots,s\}$.  In other words,  $f_*(d_1)=\ldots=f_*(d_s)=\pm 1$, and so    
\begin{equation}\label{eq:gradocur1}
|\sum_{j=1}^s f_*(d_j)|=s=\#(X^{-1}(L)\cap \overline \Omega).
\end{equation}
Likewise $f_*(u_j)=\pm I_{q_j}$ for all $j\in \{1,\ldots,r\}$, where the sign does not depend on $j$, hence 
\begin{equation}\label{eq:gradocur2}
|\sum_{j=1}^r f_*(u_j)| = \sum_{j=1}^r I_{q_j}.
\end{equation}

Let us obtain an estimation of $|\sum_{j=1}^m f_*(c_j)|$.

For each $j\in\{1,\ldots,m\}$ consider the closed planar  oriented  curve $\alpha_j:=(X_1,X_2)(c_j)$.

 From Remark \ref{re:tv}, the analytical curve $\alpha_j$ is   piecewise regular.
 Moreover, $\alpha_j$ admits a regular normal field which,  up to the identification $\s^1\equiv \s^2\cap\{x_3=0\}$,  coincides with  $\pm N|_{c_j}$. Therefore, if we call $w_j$ and $t_j$ the winding number with respect to the origin and the turning number of $\alpha_j$, respectively,   Proposition \ref{pro:curve} gives that  $2t_j\geq |w_j|$. Since $\alpha_j$ and $f(c_j)=\alpha_j/\|\alpha_j\|$ have the same winding number with respect to the origin, then $w_j= f_*(c_j),$ and so  
\begin{equation}\label{eq:45}
	|\sum_{j=1}^m f_*(c_j)|\le \sum_{j=1}^m |w_j| \le  \sum_{j=1}^m 2t_j.
\end{equation}
On the other hand, we know that $N(\overline{\Omega})$  is either the  Northern  or Southern closed hemisphere. Furthermore, 
   $ N|_{\overline{\Omega}}\colon \overline{\Omega}\to N(\overline{\Omega}) $ is a finite branched covering of degree ${\rm Deg}(N|_{\overline \Omega})\le{\rm Deg}(N)$. Since   the normal vector field to the planar curve  $\alpha_j $ is a regular map that coincides,  up to the sign,  with  $N|_{c_j}$,   $t_j$ is equal to the topological degree of   $N|_{c_j}\colon c_j \to \s^1$, $j=1,\ldots,m$, and $\sum_{j=1}^m t_j= {\rm Deg}(N|_{\overline \Omega})$. In view of \eqref{eq:45}, we infer that
\begin{equation}\label{eq:gradocur3}
|\sum_{j=1}^m f_*(c_j)|\leq 2{\rm Deg}(N|_{\overline\Omega}).
\end{equation} 
 By using equations \eqref{eq:gradocur}, \eqref{eq:gradocur1}, \eqref{eq:gradocur2}, and \eqref{eq:gradocur3}, we get that
\[ 
\#(X^{-1}(L)\cap \overline \Omega)\leq \sum_{j=1}^r I_{q_j}+2{\rm Deg}(N|_{\overline\Omega}).
\]

Joining together this information for all components  $\Omega$ of $\Sigma\setminus Q$, and taking into account that each Jordan curve in $Q$ lies in the boundary of exactly two of these components, we get that
\begin{equation}\label{eq:fin}
 \#(X^{-1}(L))\leq \sum_{q\in E} I_{q}+4{\rm Deg}(N).
\end{equation}
 On the other hand, Jorge-Meeks formula \eqref{eq:j-m} says that  $\sum_{q\in E} I_{q}=2{\rm Deg}(N)+\chi (M)$, hence from \eqref{eq:fin},
  $\# \big(X^{-1}(L)\big)\leq  
  6 {\rm Deg}(N)+\chi(M) $. This concludes the proof.
\end{proof}

\begin{proof}[Proof of Theorem \ref{th:main-intro-2}]
Take integers $r\ge 1$ and $2-2r\le m\le 1$.

Choose two non-parallel coplanar affine lines $L_1$ and $L_2$ making an angle of $2 \pi a$, $a\notin \Q$, at $x_0:=L_1\cap L_2$.
 Choose $C_j\subset L_j$ such that $x_0\in C_j$ and 
 \begin{equation}\label{eq:CC}
 	\# C_j=6r+m+1,\quad j=1,2,
\end{equation}
and set $A_{r;m}:=C_1\cup C_2$. We claim that $A_{r;m}$ satisfies the conclusion of the theorem. Indeed, it is clear that $\# A_{r;m}=12r+2m+1$. Reason by contradiction and suppose that there is $X\colon M\to \r^3$ in $\bigcup_{k\le m} \Zscr_{r;k}$ with $A_{r;m}\subset X(M)$ (see Definition \ref{de:espacios}). In particular, $X$ is nonflat. By \eqref{eq:CC}, Theorem \ref{th:against} yields that $L_1\cup L_2\subset X(M)$ (in particular, $X(M)$ is not a catenoid; the catenoid does not contain any affine line), hence by Schwarz's reflection principle $X(M)$ is invariant under the reflection $\Rcal_j\colon \r^3\to \r^3$ about $L_j$, $j=1,2$. Since $a\notin \Q$, the surface $X(M)$ is invariant under an infinite group of symmetries (the one generated by $\Rcal_j$, $j=1,2$). Since $X$ is of FTC and $X(M)$ is not a catenoid, Proposition \ref{pro:isomet} implies that $X$ is flat, a contradiction.
\end{proof}

\begin{corollary}\label{co:against}
Let $r$ and $m$ be as in Theorem \ref{th:main-intro-2}. There is a set $A^*_{r;m}\subset\r^3$, consisting of $12r +2m+2$ points,  such that if $X\colon M\to\r^3$ is a complete orientable immersed minimal surface with empty boundary and with $\chi(M)\le m$ and  $A^*_{r;m}\subset X(M)$, then the absolute value of the total curvature $|{\rm TC}(X)|>4\pi r$.
\end{corollary}
\begin{proof} To construct $A^*_{r;m}$ it suffices to add to $A_{r;m}$ a point not contained in the affine plane generated by $A_{r;m}$.
\end{proof}

\begin{corollary} \label{co:Z3}
Let $\Fcal$ be a family of affine lines in $\r^3$ such that the reflections about the lines in $\Fcal$ generate an infinite group of rigid motions. For each $L\in \Fcal$ choose an  infinite subset $A_L\subset L$ and set $A:=\bigcup_{L\in \Fcal} A_L$. Then, the set $A$ is against the family of all complete nonflat  minimal surfaces in $\r^3$ with finite total curvature; i.e., there is no such surface containing $A$.

In particular, $\z^3$ is against the mentioned family of surfaces.
\end{corollary}
\begin{proof} 
Reason by contradiction and suppose that there is a complete minimal surface with  $X\colon M\to \r^3$ such that $A\subset X(M)$. By Theorem \ref{th:against}, $\bigcup_{L\in \Fcal} L\subset X(M)$, hence $X(M)$ is invariant by the group of rigid motions generated by the reflections about these lines. Since the catenoid contains no affine lines,  Proposition \ref{pro:isomet} shows that $X$ is flat, a contradiction.
\end{proof}

It is perhaps worth mentioning that, by the results in \cite{AlarconCastro2017}, there are complete minimal surfaces in $\r^3$ containing a set $A$ as in the corollary. By Corollary \ref{co:Z3}, that surfaces are of infinite total curvature.

%%%%%%%%%%
%%%%%%%%%%
%%%%%%%%%%
%%%%%%%%%%   THANKS
%%%%%%%%%%
%%%%%%%%%%

\subsection*{Acknowledgements}
The authors were partially supported by the State Research Agency (SRA) and European Regional Development Fund (ERDF) via the grants no. MTM2014-52368-P and MTM2017-89677-P, MICINN, Spain.
They wish to thank an anonymous referee for valuable suggestions which led to an improvement of the exposition.

%%%%%%%%%%
%%%%%%%%%%
%%%%%%%%%%
%%%%%%%%%%   THE BIBLIOGRAPHY
%%%%%%%%%%
%%%%%%%%%%

%%%%%%%%%%
%%%%%%%%%%
%%%%%%%%%%
%%%%%%%%%%   AFFILIATIONS
%%%%%%%%%%
%%%%%%%%%%

\bigskip

\noindent Antonio Alarc\'{o}n

\noindent Departamento de Geometr\'{\i}a y Topolog\'{\i}a e Instituto de Matem\'aticas (IEMath-GR), Universidad de Granada, Campus de Fuentenueva s/n, E--18071 Granada, Spain.

\noindent  e-mail: {\tt alarcon@ugr.es}

\bigskip

\noindent Ildefonso Castro-Infantes

\noindent Departamento de Geometr\'{\i}a y Topolog\'{\i}a e Instituto de Matem\'aticas (IEMath-GR), Universidad de Granada, Campus de Fuentenueva s/n, E--18071 Granada, Spain.

\noindent  e-mail: {\tt icastroinfantes@ugr.es}

\bigskip

\noindent Francisco J.\ L\'opez

\noindent Departamento de Geometr\'{\i}a y Topolog\'{\i}a e Instituto de Matem\'aticas (IEMath-GR), Universidad de Granada, Campus de Fuentenueva s/n, E--18071 Granada, Spain.

\noindent  e-mail: {\tt fjlopez@ugr.es}

\end{document}